\documentclass[12pt]{amsart}
\usepackage{amssymb,amsmath,tabularx,mathrsfs,mathbbol,yfonts,upgreek}
\usepackage{amsthm,verbatim,comment}
\usepackage[bookmarks=true]{hyperref}
\usepackage{pstricks,pst-node,pst-plot}
\usepackage{geometry}
\usepackage{stmaryrd}
\usepackage{paralist}
\usepackage[all]{xy}
\numberwithin{equation}{section}

\DeclareSymbolFontAlphabet{\mathbb}{AMSb}
\DeclareSymbolFontAlphabet{\mathbbl}{bbold}

\newtheorem{thm}{Theorem}[section]
\newtheorem{lem}[thm]{Lemma}
\newtheorem{prop}[thm]{Proposition}
\newtheorem{cor}[thm]{Corollary}

\theoremstyle{definition}
\newtheorem{defn}[thm]{Definition}

\theoremstyle{remarks}
\newtheorem*{rem*}{Remarks}

\renewcommand{\t}{\mathfrak{t}}
\newcommand{\s}{\mathfrak{s}}

\newcommand{\Tr}{\mathrm{Tr}}

\newcommand{\sgn}{\mathrm{sgn}}

\newcommand{\orbit}[2]{#1/#2}

\newcommand{\sym}[1]{\mathfrak{S}_{#1}}

\newcommand{\Inf}{\mathrm{Inf}}

\newcommand{\Y}{\mathrm{Y}}
\newcommand{\C}{\mathscr{C}}
\newcommand{\NN}{\mathbb{N}}
\renewcommand{\O}{\mathcal{O}}
\newcommand{\M}{\mathrm{M}}

\newcommand{\N}{\mathrm{N}}
\renewcommand{\P}{\mathscr{P}}
\newcommand{\RP}{\mathscr{RP}_p}
\newcommand{\Kostka}{\mathrm{K}}

\newcommand{\Def}{\mathrm{Def}}

\newcommand{\ch}{\mathrm{ch}}
\newcommand{\m}{\mathbf{m}}
\newcommand{\set}[1]{[#1]}

\newcommand{\B}{\mathcal{B}}

\newcommand{\F}{\mathbb{F}}
\newcommand{\K}{\mathbb{K}}

\title[Brauer constructions and generic Jordan types]{On the Brauer constructions and generic Jordan types of Young modules}
\begin{document}
\author{Yu Jiang}
\address[Y. Jiang]{Division of Mathematical Sciences, Nanyang Technological University, SPMS-MAS-05-34, 21 Nanyang Link, Singapore 637371.}
\email[Y. Jiang]{jian0089@e.ntu.edu.sg}

\author{Kay Jin Lim}
\address[K. J. Lim]{Division of Mathematical Sciences, School of Physical and Mathematical Sciences, Nanyang Technological University, 21 Nanyang Link, Singapore 637371.}
\email[K. J. Lim]{limkj@ntu.edu.sg}

\author{Jialin Wang}
\address[JL. Wang]{Division of Mathematical Sciences, Nanyang Technological University, SPMS-MAS-05-34, 21 Nanyang Link, Singapore 637371.}
\email[JL. Wang]{wangjl@ntu.edu.sg}

\thanks{The second and third authors are supported by the Singapore Ministry of Education AcRF Tier 1 grant RG13/14.}

\begin{abstract} Let $p$ be a prime number. We study the dimensions of the Brauer constructions of Young and Young permutation modules with respect to $p$-subgroups of the symmetric groups. They depend only on partitions labelling the modules and the orbits of the action of the $p$-subgroups, and are related to their generic Jordan types. We obtain some reductive formulae and, in the case of two-part partitions, make some explicit calculation.
\end{abstract}

\maketitle

\section{introduction}
Let $G$ be a finite group and $\F$ be a field of positive characteristic $p$. One of the main tools of studying the $p$-permutation $\F G$-modules via the Brauer construction has been developed by Brou\'{e} in \cite{MBroue}. For a $p$-subgroup $P$ of $G$, there is a bijection between the set of the isomorphism classes of indecomposable $p$-permutation $\F G$-modules with vertex $P$ and the set of isomorphism classes of indecomposable projective $\F[\N_G(P)/P]$-modules defined by the Brauer construction. Furthermore, the Green correspondents of such indecomposable $p$-permutation $\F G$-modules with respect to the subgroup $\N_G(P)$ are precisely the inflation of their corresponding indecomposable projective $\F[\N_G(P)/P]$-modules. Suppose further that $\F$ is algebraically closed. The generic Jordan type of a module for an elementary Abelian $p$-group as defined by Wheeler \cite{WW} is another useful technique to study the $\F G$-modules. For instance, if an indecomposable $\F G$-module $M$ has non-generically free Jordan type upon restriction to an elementary Abelian $p$-subgroup $E$ of $G$ then $E$ is contained in a vertex of $M$.

In this paper, we study the classical objects the Young and Young permutation $\F\sym{n}$-modules. Since they are $p$-permutation modules, their stable generic Jordan types (modulo the projectives) restricted to any elementary Abelian $p$-subgroup $E$ of $\sym{n}$ have the form $[1]^r$ for some non-negative integers $r$ depending on $E$. We are interested in the numbers $r$ as in the previous sentence. In Section \ref{S: orbit numbers}, one of our main results shows that the dimension of the Brauer construction $Y^\lambda(E)$ is precisely $r$ where $[1]^r$ is the stable generic Jordan type of $Y^\lambda{\downarrow_E}$ and $\dim_\F Y^\lambda(P)$, for any $p$-subgroup $P$ of $\sym{n}$, depend only on the orbit type of $P$ on the set $\{1,\ldots,n\}$. As such, we call $\dim_\F Y^\lambda(P)$ the orbit numbers. For example, when $n=4$, $P=\langle(1,2,3,4)\rangle$ and $Q=\langle (1,2)(3,4),(1,3)(2,4)\rangle$, for any partition $\lambda$ of $4$, we have $\dim_\F Y^\lambda(P)=\dim_\F Y^\lambda(Q)=r$ where $Y^\lambda{\downarrow_E}$ has stable generic Jordan type $[1]^r$. The orbit numbers are interesting in the sense that, when $P$ is a vertex of $Y^\lambda$, following \cite[Theorem 2]{KErdmann}, we have $\dim_\F Y^\lambda(P)$ is the product of the dimensions of the projective modules $Y^{\lambda(0)},\ldots,Y^{\lambda(s)}$ where $\lambda=\lambda(0)+p\lambda(1)+\cdots+p^s\lambda(s)$ is the $p$-adic expansion of $\lambda$. It is an open problem to find a closed-form for the dimensions of the indecomposable projective modules for the symmetric groups. In Section \ref{S: Some computation}, we obtain some reductive formulae about orbit numbers. We explicitly calculate the orbit numbers in the case when $\lambda$ are two-part partitions in Section \ref{S: two part}.

\section{preliminaries}
Throughout the paper $\F$ is an algebraically closed field of positive characteristic $p$. For any finite group $G$, an $\F G$-module is assumed to be a finitely generated left $\F G$-module.

\subsection{Representation theory of finite groups}
For a general background about the modular representation theory of finite groups, we refer readers to \cite{JAlperin} or \cite{HNYT}.

Let $G$ be a finite group and let $M,N$ be two $\F G$-modules. We write $N\mid M$ if $N$ is isomorphic to a direct summand of $M$, i.e., $M\cong N\oplus L$ for some $\F G$-module $L$. Suppose further that $N$ is indecomposable. The number of summands in an indecomposable direct sum decomposition of $M$ that are isomorphic to $N$ is well-defined by Krull-Schmidt Theorem (see \cite[Section 4, Theorem 3]{JAlperin}) and is denoted by $[M: N]$.

Let $M$ be an indecomposable $\F G$-module and $H$ be a subgroup of $G$. Then $M$ is said to be relatively $H$-projective if there exists some $\F H$-module $N$ such that $M\mid N{\uparrow^{G}}$, here $N{\uparrow^G}$ denotes the induction of $N$ to $G$. By \cite{JGreen}, the minimal (with respect to inclusion of subgroups) subgroups $H$ of $G$ subject to the condition such that $M\mid N{\uparrow^G}$ for some $\F H$-module $N$ are $p$-subgroups and unique up to $G$-conjugation. These $p$-subgroups of $G$ are called the vertices of $M$. Let $P$ be a vertex of $M$. We denote the normalizer of $P$ in $G$ by $\N_G(P)$. Then there exists, unique up to isomorphism and $\N_G(P)$-conjugation, an indecomposable $\F P$-module $S$ such that $M\mid S{\uparrow^{G}}$. Such an $\F P$-module is called a source of $M$.

Let $M$ be an indecomposable $\F G$-module, let $P$ be a vertex of $M$ and let $H$ be a subgroup of $G$ containing $\N_G(P)$. The Green correspondent of $M$ with respect to the subgroup $H$ is the unique indecomposable summand $N$ of $M{\downarrow_H}$ such that $N$ has a vertex $P$.

Let $E=\langle g_1,\ldots,g_k\rangle$ be an elementary Abelian $p$-group of order $p^k$ and $M$ be an $\F E$-module. Let $\K$ be a field extension of $\F$ containing the indeterminates $\alpha_{1}, \alpha_{2},\ldots, \alpha_{k}$. Consider the element
\[u_{\alpha}:=1+\sum_{i=1}^k\alpha_{i}(g_{i}-1)\in \K E.\] Since $\langle u_\alpha\rangle$ is a cyclic group of order $p$, the restriction of $\K\otimes_{\F}M$ to the shifted subgroup $\K\langle u_\alpha\rangle$ is isomorphic to a direct sum of $n_i$ unipotent Jordan blocks of sizes $i$ where $i=1,2,\ldots, p$. The generic Jordan type of the $\F E$-module $M$ is defined as $[1]^{n_{1}}[2]^{n_{2}}\cdots [p]^{n_{p}}$. By \cite{WW}, the generic Jordan type is independent of the choice of the generators of $E$. The stable generic Jordan type of $M$ is $[1]^{n_{1}}[2]^{n_{2}}\cdots [p-1]^{n_{p-1}}$. The module $M$ is called generically free if $n_i=0$ for all $i=1,2,\ldots, p-1$. The following are the properties we shall need and we refer readers to \cite{EFJPAS,GLW,WW} for more details.

\begin{lem}\label{L: generic Jordan type}\
\begin{enumerate}[(i)]
 \item The generic Jordan type of a direct sum of modules is the direct sum of the generic Jordan types of the modules.
 \item Let $E'$ be a proper subgroup of an elementary Abelian $p$-group $E$ and let $M$ be an $\F E'$-module. Then the module $M{\uparrow^{E}}$ is generically free.
\end{enumerate}
\end{lem}

\subsection{Brauer construction}
One of main techniques that we shall need is the Brauer constructions of $p$-permutation modules introduced by Brou\'{e} in \cite{MBroue}. An $\F G$-module is called a $p$-permutation module if for any $p$-subgroup $P$ of $G$ there exists a basis $\B$ that is permuted by $P$, i.e., for each $g\in P$ and $b\in \B$, we have $gb\in\B$. By \cite[(0.4)]{MBroue}, an indecomposable $p$-permutation module is precisely a module with trivial source. The class of all $p$-permutation $\F G$-modules is closed under taking finite direct sum, direct summand and tensor product.

We recall the Brauer construction of a module. Let $M$ be an $\F G$-module and $P$ be a $p$-subgroup of $G$. The set of $P$-fixed points in $M$ is \[M^{P}:=\{m\in M:\text{$gm=m$ for all $g\in P$}\}.\] Notice that $M^{P}$ is an $\F\N_{G}(P)$-module on which $P$ acts trivially. Let $Q$ be a proper subgroup of $P$. The relative trace map $\Tr_{Q}^{P}$: $M^{Q}\rightarrow M^{P}$ is the linear map defined by \[\Tr_{Q}^{P}(m):=\sum_{g\in P/Q}gm,\] where $P/Q$ denotes a set of left coset representatives of $Q$ in $P$ and $m\in M^{Q}$. Observe that $\Tr_{Q}^{P}(v)$ is independent of the choice of the set of left coset representatives. Furthermore \[\Tr^{P}(M):=\sum \Tr_{Q}^{P}(M^Q),\] where $Q$ runs over the set of all proper subgroups of $P$, is an $\F \N_{G}(P)$-submodule of $M^{P}$. One defines the Brauer construction of $M$ with respect to $P$ to be the $\F [\N_{G}(P)/P]$-module \[M(P):=M^{P}/\Tr^{P}(M).\] In general, if $M$ is indecomposable and $M(P)\neq 0$ then $P$ is contained in a vertex of $M$. The converse is true in the case of $p$-permutation modules.

\begin{thm}[{\cite[Theorem 3.2 (1)]{MBroue}}]\label{Contain}
Let $M$ be an indecomposable $p$-permutation $\F G$-module, let $Q$ be a vertex of $M$ and let $P$ be a $p$-subgroup of $G$. Then $M(P)\neq 0$ if and only if $P$ is contained in a $G$-conjugate of $Q$.
\end{thm}

Suppose further that $M$ is a $p$-permutation $\F G$-module. Let $\B$ be a basis of $M$ permuted by $P$ and let
\begin{align*}
\B(P):=\{b\in \B:\text{$gb=b$ for all $g\in P$}\}.
\end{align*} Notice that $P$ acts trivially on $\B(P)$. As a corollary of Theorem \ref{Contain}, we have the following.

\begin{cor}\label{C: EG}
Let $M$ be a $p$-permutation $\F G$-module and $\B$ be a $p$-permutation basis of $M$ with respect to a $p$-subgroup $P$ of $G$. Then $M(P)$ is isomorphic to the $\F$-span of $\B(P)$ as $\F [\N_{G}(P)/P]$-modules. Furthermore, if $M\cong U\oplus V$ then \[M(P)\cong U(P)\oplus V(P).\]
\end{cor}

We end this subsection with the following well-known result of Brou\'{e}.

\begin{thm}[{\cite[Theorems 3.2 and 3.4]{MBroue}}]\label{BC1} Let $G$ be a finite group and let $P$ be a $p$-subgroup of $G$.
\begin{enumerate}
\item [(i)] The Brauer construction sending $M$ to $M(P)$ is a bijection between the isomorphism classes of indecomposable $p$-permutation $\F G$-modules with vertex $P$ and the isomorphism classes of indecomposable projective $\F[\N_G(P)/P]$-modules. Furthermore, the inflation $\Inf_{\N_G(P)/P}^{\N_G(P)}M(P)$ of the $\F[\N_G(P)/P]$-module $M(P)$ to $\N_G(P)$ is the Green correspondent of~$M$ with respect to $\N_G(P)$.
\item [(ii)] Let $N$ be an indecomposable $\F G$-module with a vertex $P$ and $M$ be a $p$-permutation $\F G$-module. Then $N$ is a direct summand of $M$ if and only if $N(P)$ is a direct summand of $M(P)$. Moreover, \[[M:N]=[M(P):N(P)].\]
\end{enumerate}
\end{thm}

\subsection{Composition, partition and orbit}\label{SS: composition} Let $\NN$ be the set of nonnegative integers and let $n\in\NN$. By a composition $\alpha$ of $n$, we mean a sequence of nonnegative integers $(\alpha_1,\ldots,\alpha_r)$ such that $\sum^r_{i=1}\alpha_i=n$. In this case, we write $|\alpha|=n$. By convention, the unique composition of $0$ is denoted as $\varnothing$. The composition $\alpha$ is called a partition if $\alpha_1\geq\cdots\geq\alpha_r$. We write $\C(n)$ and $\P(n)$ for the set of compositions and partitions of $n$ respectively. The set $\P(n)$ is partially ordered by the dominance order $\unrhd$ and totally ordered by the lexicographic order. Notice that the lexicographic order refines the dominance order.


Let $\alpha=(\alpha_1,\ldots,\alpha_r)$ and $\beta=(\beta_1,\ldots,\beta_s)$ be two compositions and let $m$ be a positive integer. We write \begin{align*}
\alpha+\beta=\beta+\alpha&=(\alpha_1+\beta_1,\ldots,\alpha_r+\beta_r,\beta_{r+1},\ldots,\beta_s),\\
\alpha\bullet \beta&=(\alpha_1,\ldots,\alpha_r,\beta_1,\ldots,\beta_s),\\
m\alpha&=(m\alpha_1,\ldots,m\alpha_r),
\end{align*} if $r\leq s$. A composition $\delta$ is a refinement of $\beta$ if there exist compositions $\delta^{(1)},\ldots,\delta^{(s)}$ such that $\delta=\delta^{(1)}\bullet\cdots\bullet\delta^{(s)}$ and, for $i=1,2,\ldots, s$, we have $|\delta^{(i)}|=\beta_i$.

A partition $\lambda=(\lambda_1,\ldots,\lambda_r)$ is called $p$-restricted if $\lambda_r<p$ and $\lambda_i-\lambda_{i+1}<p$ for all $i=1,2,\ldots,r-1$. We write $\RP(n)$ for the set of all $p$-restricted partitions of $n$. The $p$-adic expansion of a partition $\lambda$ is the sum \[\lambda=\sum^t_{i=0}p^{i}\lambda(i)\] for some nonnegative integer $t$ such that, for each $i=0,1,\ldots,t$, $\lambda(i)$ is a $p$-restricted partition. By the proof of \cite[Lemma 7.5]{GJR}, there is a way to write down the $p$-adic expansion of $\lambda$ as follows. Let $\lambda_{r+1}=0$. Suppose that, for each $j=1,2,\ldots, r$, we have the $p$-adic sum of the number \[\lambda_j-\lambda_{j+1}=\sum_{i= 0}^{t}a_{i,j}p^i,\] i.e., $0\leq a_{i,j}\leq p-1$. Then, for each $i=0,1,\ldots,t$, $\lambda(i)$ is the desired $p$-restricted partition where $\lambda(i)_k=\sum^r_{j=k}a_{i,j}$.

For any partition $\lambda=(\lambda_1,\ldots,\lambda_r)$, we denote by $[\lambda]$ the set $\{(i,j)\in \NN^{2}:\ 1\leq i\leq r,\  1\leq j\leq \lambda_{i}\}$. It is called the Young diagram of $\lambda$. The $p$-core of $\lambda$ is the partition whose Young diagram is obtained by removing all possible rim $p$-hooks from $[\lambda]$ and is denoted by $\kappa_{p}(\lambda)$. The number of rim $p$-hooks removed from $[\lambda]$ to get $\kappa_{p}(\lambda)$ is called the $p$-weight of $\lambda$.

Let $A$ be a finite set. The permutation group on the set $A$ is denoted as $\sym{A}$. Let $n\in\NN$. We denote the set $\{1,\ldots,n\}$ by $\set{n}$ and let $\sym{n}=\sym{[n]}$. By convention, $\set{0}=\emptyset$ and $\sym{0}$ is the trivial group. Let $\lambda=(\lambda_1,\ldots,\lambda_r)$ be a composition. The Young subgroup $\sym{\lambda}$ is identified with the direct product \[\sym{\lambda_1}\times\cdots\times\sym{\lambda_r},\] where the first factor $\sym{\lambda_1}$ acts on the set $\{1,\ldots,\lambda_1\}$, the second factor $\sym{\lambda_2}$ acts on the set $\{1+\lambda_1,\ldots,\lambda_1+\lambda_2\}$ and so on.

We now discuss the orbits of $p$-subgroups of $\sym{n}$ on the set $\set{n}$. Let \[\P^{(p)}(n)=\{\O\in \C(n): \text{$\O=(1^{a_{0}}, p^{a_{1}},\ldots, (p^{r})^{a_{r}})$ for some $r$}\},\] where, in $\O$, the first $a_0$ entries of $\O$ are $1$, the next $a_1$ entries are $p$ and so on. For each $\O=(1^{a_{0}}, p^{a_{1}},\ldots, (p^{r})^{a_{r}})\in\P^{(p)}(n)$, let \[p^s\O=((p^s)^{a_0},(p^{s+1})^{a_1},\ldots,(p^{s+r})^{a_r})\in\P^{(p)}(p^sn).\] Let $P$ be a $p$-subgroup of $\sym{n}$. We denote the set of orbits of the action of $P$ on $\set{n}$ by $\orbit{\set{n}}{P}$. We say that $\orbit{\set{n}}{P}$ has type $\O=(1^{a_0},(p)^{a_1},\ldots,(p^r)^{a_r})\in\P^{(p)}(n)$ for some $r$ if, for each $i=0,1,\ldots,r$, the number of orbits with sizes $p^i$ in $\orbit{\set{n}}{P}$ is exactly $a_i$. We write $\orbit{\set{n}}{P}\simeq \orbit{\set{n}}{Q}$ if $Q$ is another $p$-subgroup of $\sym{n}$ such that both $\orbit{\set{n}}{P}$ and $\orbit{\set{n}}{Q}$ have the same type, i.e., there is a permutation $\sigma\in\sym{n}$ such that $\sigma A\in \orbit{\set{n}}{Q}$ for all $A\in \orbit{\set{n}}{P}$. It is clear that $\orbit{\set{n}}{P}\simeq \orbit{\set{n}}{Q}$ if $P$ is conjugate to $Q$ in $\sym{n}$.

Let $\lambda$ be a partition of $n$, let $\sum_{i=0}^{r}p^{i}\lambda(i)$ be the $p$-adic expansion of $\lambda$ and let \[\O_\lambda:=(1^{|\lambda(0)|}, p^{|\lambda(1)|},\ldots, (p^{r})^{|\lambda(r)|})\in\P^{(p)}(n).\] We fix a Sylow $p$-subgroup of $\sym{\O_\lambda}$ and denote it by $P_\lambda$. Notice that, since any Sylow $p$-subgroup of $\sym{p^i}$ acts transitively on the set $\set{p^i}$, we have that $\orbit{\set{n}}{P_\lambda}$ has type $\O_\lambda$. We end this subsection by the following lemma.

\begin{lem}\label{L: minimal subgroup}\
Let $\lambda\in\P(n)$, $\O\in\P^{(p)}(n)$ and $P$ be a $p$-subgroup of $\sym{n}$ such that $\orbit{\set{n}}{P}$ has type $\O$. Then $\O$ is a rearrangement of a refinement of $\O_\lambda$ if and only if $P$ is conjugate to a $p$-subgroup of $P_\lambda$.
\end{lem}
\begin{proof} Suppose that $\O$ is a rearrangement of a refinement of $\O_\lambda$. Without loss of generality, since $\orbit{\set{n}}{P_\lambda}$ has type $\O_\lambda$, we may assume that each orbit in $\orbit{\set{n}}{P_\lambda}$ is a union of some orbits in $\orbit{\set{n}}{P}$. Let $\sigma\in P$. Then $\sigma$ leaves each orbit $A$ in $\orbit{\set{n}}{P}$ invariant, i.e., $\sigma(A)=A$. Therefore, $\sigma$ leaves each orbit in $\orbit{\set{n}}{P_\lambda}$ invariant. This shows that $\sigma\in\sym{\O_\lambda}$ and hence $P\subseteq \sym{\O_\lambda}$. We conclude that $P$ is conjugate to a subgroup of $P_\lambda$. Conversely, suppose, without loss of generality, that $P$ is a subgroup of $P_\lambda$. Then each orbit in $\orbit{\set{n}}{P_\lambda}$ is a union of some orbits in $\orbit{\set{n}}{P}$. By definition, the type $\O$ of $\orbit{\set{n}}{P}$ is a rearrangement of a refinement of the type $\O_\lambda$ of $\orbit{\set{n}}{P_\lambda}$. 
\end{proof}

\subsection{Representation theory of symmetric groups}\label{SS: sym}  We now turn to the representation theory of symmetric groups. For a general background on this topic, we refer readers to \cite{GJ1} or \cite{GJ3}.

For modules of symmetric groups, we assume that readers are familiar with the notion of tableau, tabloid and polytabloid. Let $\F(n)$ be the trivial $\F\sym{n}$-module. For a composition $\lambda$ of $n$, we use $\F(\lambda)$ to denote the restriction of $\F(n)$ to the Young subgroup $\sym{\lambda}$.  The Young permutation module $M^{\lambda}$ with respect to $\lambda$ is the induced module $\F(\lambda){\uparrow^{\sym{n}}}$. It has a basis consisting of all $\lambda$-tabloids. Notice that $M^\lambda\cong M^\mu$ if $\mu$ can be rearranged to $\lambda$. Suppose now that $\lambda$ is a partition. The Specht module $S^\lambda$ is the submodule of $M^\lambda$ spanned by the $\lambda$-polytabloids. It has a basis given by the standard $\lambda$-polytabloids and dimension given by the hook formula. In the characteristic zero case, the Specht modules are the irreducible $\F\sym{n}$-modules. However, they are usually not irreducible when $p$ is positive. By the Nakayama conjecture, two Specht modules $S^\lambda,S^\mu$ for $\F\sym{n}$ lie in the same block if and only if $\kappa_p(\lambda)=\kappa_p(\mu)$. 

The isomorphism classes of indecomposable direct summands of Young permutation modules are called the Young modules and they are parametrized by $\P(n)$ such that the Young module $Y^\lambda$ is a direct summand of $M^\lambda$ with multiplicity one and, if $Y^\mu\mid M^\lambda$, then $\mu\trianglerighteq \lambda$ (see \cite[Theorem 3.1]{GJ2}), i.e., \[M^{\lambda}\cong Y^{\lambda}\oplus\bigoplus_{\mu\vartriangleright\lambda}k_{\lambda,\mu}Y^{\mu},\] where $k_{\lambda,\mu}=[M^{\lambda}: Y^{\mu}]$ are known as the $p$-Kostka numbers. Using the lexicographic order of $\P(n)$, we denote the $p$-Kostka matrix for $\F\sym{n}$ by $\Kostka$ whose $(\lambda,\mu)$-entry is $k_{\lambda,\mu}$. Notice that $\Kostka$ is upper uni-triangular.

We recall the following reductive formulae for $p$-Kostka numbers proved by Gill.
\begin{thm}[{see \cite[Theorems 13 and 14]{CGill}}]\label{T: Reductive}
Let $\lambda, \mu\in\P(m)$ and $\nu, \delta\in\P(n)$. We have the following statements. 
\begin{enumerate}[(i)]
 \item Let $\lambda_{1}$ be the first part of $\lambda$ and $\sum_{i=0}^{t}p^{i}\mu(i)$ be the $p$-adic expansion of $\mu$. If $s>t$ and $p^{s}>\lambda_{1}$ then $k_{\lambda+p^{s}\nu,\mu+p^{s}\delta}=k_{\lambda,\mu}k_{\nu,\delta}$.
 \item Let $\lambda_{2}$ be the second part of $\lambda$. If $\lambda_{2}<p^{s}$ then $k_{\lambda+(p^{s}r),\mu+(p^{s}r)}=k_{\lambda,\mu}$ for every $r\in \NN$.
 \end{enumerate}
\end{thm}

The regular module $\F\sym{n}$ is the Young permutation module $M^{(1^n)}$. Therefore the projective indecomposable $\F\sym{n}$-modules are Young modules. In fact, the Young module $Y^\lambda$ is projective if and only if $\lambda\in\RP(n)$. Let $\lambda\in\RP(n)$ and $\sgn(n)$ be the signature representation of $\F\sym{n}$. Since $Y^\lambda\otimes\sgn(n)$ is also projective indecomposable, we have \[Y^\lambda\otimes\sgn(n)\cong Y^{\m(\lambda)}\] for some unique partition $\m(\lambda)\in\RP(n)$. The map $\m:\RP(n)\to\RP(n)$ is called the Mullineux map (on $p$-restricted partitions) and is an involution. The $p$-regular version of Mullineux map was conjectured by Mullineux in \cite{GMullineux} and proved by Ford and Kleshchev in \cite{FordKleshchev}. In \cite{BrKu}, Brundan and Kujawa proved the $p$-restricted version.


We now discuss the Brauer constructions of Young permutation modules and Young modules as in \cite{KErdmann}.

For each $\lambda\in\P(n)$ and $\O\in\P^{(p)}(n)$, let $P$ be a $p$-subgroup of $\sym{n}$ such that $\orbit{\set{n}}{P}$ has type $\O$. Let $M_{\lambda,P}$ be the set of all $\lambda$-tabloids $\t$ such that each row of $\t$ is a union of some orbits in $[n]/P$. Notice that if $Q$ is another $p$-subgroup of $\sym{n}$ such that $\orbit{\set{n}}{Q}\simeq \orbit{\set{n}}{P}$ then $|M_{\lambda,P}|=|M_{\lambda,Q}|$. We write \[m_{\lambda,\O}=|M_{\lambda,P}|.\]

Since, for each $\lambda\in\P(n)$, the Young permutation module $M^\lambda$ has basis the $\lambda$-tabloids permuted by $\sym{n}$, and hence permuted by any $p$-subgroup of $\sym{n}$, Young permutation and Young modules are $p$-permutation $\F \sym{n}$-modules. Let $P$ be a $p$-subgroup of $\sym{n}$. Notice that a $\lambda$-tabloid $\t$ is fixed by $P$ if and only if every orbit in $\set{n}/P$ lies in a row of $\t$. By Corollary \ref{C: EG}, we have the following lemma.

\begin{lem}\label{L: dim of M(P)} Let $\lambda\in\P(n)$, let $P$ be a $p$-subgroup of $\sym{n}$ and suppose that $\orbit{\set{n}}{P}$ has type $\O$. Then \[\dim_\F M^\lambda(P)=m_{\lambda,\O},\] i.e., $\dim_\F M^\lambda(P)$ is the number of (unordered) ways to insert the orbits in $\orbit{\set{n}}{P}$ into the rows of $\lambda$. In particular, we have $\dim_\F M^\lambda(P)=\dim_\F M^\lambda(Q)$ if $\orbit{\set{n}}{P}\simeq \orbit{\set{n}}{Q}$.
\end{lem}

The precise structure of the Brauer construction $M^\lambda(P)$ when $\lambda\in\P(n)$, $P$ is a Sylow $p$-subgroup of $\sym{\O}$ and $\O\in\P^{(p)}(n)$ is given in \cite[Proposition 1]{KErdmann} but we shall not need it here.

Suppose that a normal subgroup $N$ of $G$ acts trivially on an $\F G$-module $M$. We write $\Def^G_{G/N}M$ for the deflation of $M$ to the quotient group $G/N$.
We now describe the vertices of Young modules and their Brauer constructions with respect to the vertices.

\begin{thm}[{\cite{GJR,KErdmann}}]\label{T: GJR} Let $\lambda\in\P(n)$. Then $Y^{\lambda}$ is relatively $\sym{\O_\lambda}$-projective. If $Y^{\lambda}$ is also relatively $H$-projective for some Young subgroup $H$ then $\sym{\O_\lambda}$ is $\sym{n}$-conjugate to a subgroup of $H$. Furthermore, $Y^\lambda$ has a vertex $P_\lambda$.
\end{thm}

\begin{thm}[{\cite{KErdmann}}]\label{Erdmann2} Let $\sum_{i=0}^rp^i\lambda(i)$ be the $p$-adic expansion of $\lambda\in\P(n)$ and let $\beta=(a_0,a_1,\ldots,a_r)$ where $a_i=|\lambda(i)|$ for each $i=0,1,\ldots,r$. Then $\N_{\sym{\O_\lambda}}(P_\lambda)/P_\lambda$ acts trivially on $Y^\lambda(P_\lambda)$ and \[\Def^{\N_{\sym{n}}(P_\lambda)/P_\lambda}_{\N_{\sym{n}}(P_\lambda)/\N_{\sym{\O_\lambda}}(P_\lambda)}Y^{\lambda}(P_\lambda)\cong Y^{\lambda(0)}\boxtimes Y^{\lambda(1)}\boxtimes\cdots\boxtimes Y^{\lambda(r)}\] as $\F \sym{\beta}$-modules via the canonical isomorphism \[\sym{\beta}\cong \N_{\sym{n}}(P_\lambda)/\N_{\sym{\O_\lambda}}(P_\lambda)\cong (\N_{\sym{n}}(P_\lambda)/P_\lambda)/(\N_{\sym{\O_\lambda}}(P_\lambda)/P_\lambda).\]
\end{thm}

\section{orbit numbers}\label{S: orbit numbers} In this section, we define the orbit numbers (see Definition \ref{D: orbit number}) labelled by $\P(n)\times \P^{(p)}(n)$. The numbers can be simultaneously defined as either the dimensions of the Brauer constructions of Young modules with respect to $p$-subgroups or the nonnegative integers $m$ where $[1]^m$ is the stable generic Jordan types of Young modules restricted to certain elementary Abelian $p$-subgroups.

We begin with the following lemma.


\begin{lem}\label{L: sgjt dim} Let $M$ be a $p$-permutation $\F G$-module and $E$ be an elementary Abelian $p$-subgroup of $G$. Then the stable generic Jordan type of $M{\downarrow_{E}}$ is $[1]^r$ where $r=\dim_{\F}M(E)$.
\end{lem}
\begin{proof}
Let $\B$ be a $p$-permutation basis of $M$ with respect to $E$ and suppose that $A_1,\ldots,A_r,B_1,\ldots,B_s$ are the orbits of action of $E$ on $\B$ such that $|A_i|=1$ for $i=1,2,\ldots, r$ and $|B_j|>1$ for $j=1,\ldots,s$. Then \[M{\downarrow_{E}}\cong \left (\bigoplus_{i=1}^r\F_E\right )\oplus\left (\bigoplus_{j=1}^{s} \F_{H_j}{\uparrow^{E}}\right )\] as $\F E$-modules where $H_j$ is the stabiliser of $b_j\in B_j$ for all $j=1,2,\ldots, s$ and $\F_{H_j},\F_E$ are the trivial modules for $\F H_j$ and $\F E$ respectively. Since $H_j$ is a proper subgroup of $E$, by Lemma \ref{L: generic Jordan type}, $M{\downarrow_E}$ has stable generic Jordan type $[1]^r$. By Corollary \ref{C: EG}, $\dim_\F M(E)=r$. The result now follows.
\end{proof}

In the case of the Young permutation module $M^\lambda$, Lemmas \ref{L: dim of M(P)} and \ref{L: sgjt dim} assert that the generic Jordan type of $M^\lambda{\downarrow_E}$ is $[1]^r$ where \[r=\dim_\F M^\lambda(E)=m_{\lambda,\O}\] and $\orbit{\set{n}}{E}$ has type $\O$.

We now prove the main result of this section.

\begin{thm} \label{T: Yinvarant}
Let $\lambda\in\P(n)$ and $P,E$ be $p$-subgroups of $\sym{n}$ such that $\orbit{\set{n}}{P}\simeq \set{n}{E}$. Then \[\dim_\F Y^\lambda(P)=\dim_\F Y^\lambda(E).\] Suppose further that $E$ is elementary Abelian. We have $\dim_\F Y^\lambda(E)=m$ where the stable generic Jordan type of $Y^{\lambda}{\downarrow_{E}}$ is $[1]^m$.
\end{thm}
\begin{proof}
We prove that $\dim_{\F}Y^{\lambda}(P)=\dim_{F}Y^{\lambda}(E)$ by using induction on the dominance order of $\P(n)$. In the base case, since $Y^{(n)}\cong M^{(n)}$ is the trivial $\F\sym{n}$-module, we have $\dim_{\F}Y^{(n)}(P)=\dim_{\F}Y^{(n)}(E)=1$ by Lemma \ref{L: dim of M(P)}. Suppose that $\dim_\F Y^\mu(P)=\dim_\F Y^\mu(E)$ for all $\mu\rhd\lambda$. By Corollary \ref{C: EG}, we have
\begin{align*}
M^{\lambda}(P)&\cong Y^{\lambda}(P)\oplus\displaystyle{\bigoplus_{\mu\rhd\lambda}}k_{\lambda,\mu}Y^{\mu}(P),\\
M^{\lambda}(E)&\cong Y^{\lambda}(E)\oplus\displaystyle{\bigoplus_{\mu\rhd\lambda}}k_{\lambda,\mu}Y^{\mu}(E).
\end{align*} Counting the dimensions of the above equations, using Lemma \ref{L: dim of M(P)} and induction on the dominance order, we obtain that $\dim_\F Y^{\lambda}(P)=\dim_\F Y^\lambda(E)$.
Suppose further now that $E$ is elementary Abelian. Since Young modules are $p$-permutation as direct summands of Young permutation modules, the second assertion follows from Lemma \ref{L: sgjt dim}.
\end{proof}


In the view of Theorem \ref{T: Yinvarant}, we can now define the orbit number.

\begin{defn}\label{D: orbit number}
Let $\lambda\in\P(n)$, $\O\in\P^{(p)}(n)$ and let $P,E$ be $p$-subgroups of $\sym{n}$ such that both $\orbit{\set{n}}{P}$ and $\orbit{\set{n}}{E}$ have type $\O$ and $E$ is elementary Abelian. The orbit number $y_{\lambda,\O}$ is defined as the following common numbers: \[y_{\lambda,\O}:=\dim_\F Y^\lambda(P)=b,\] where $[1]^b$ is the stable generic Jordan type of $Y^\lambda{\downarrow_E}$.
\end{defn}




Fix $n\in \NN$. Let both $\P(n)$ and $\P^{(p)}(n)$ be ordered by the lexicographic order. Let $\Y,\M$ be the $\P(n)\times\P^{(p)}(n)$-matrices whose $(\lambda,\O)$-entries of $\Y,\M$ are the orbit number $y_{\lambda,\O}$ and $m_{\lambda,\O}$ respectively. By Theorem \ref{BC1}(ii), we have \begin{equation}\label{Eq: 1} m_{\lambda,\O}=y_{\lambda,\O}+\sum_{\mu\rhd \lambda}k_{\lambda,\mu}y_{\mu,\O},
\end{equation} or equivalently, $\M=\Kostka\Y$ where, recall that, $k_{\lambda,\mu}=[M^\lambda:Y^\mu]$ and $\Kostka$ is the $p$-Kostka matrix of $\F\sym{n}$. The $(\lambda,\O)$-entry $m_{\lambda,\O}$ of $\M$ has a combinatorial description given by Lemma \ref{L: dim of M(P)}. Suppose further that $\O=(1^{a_{0}}, p^{a_{1}},\ldots, (p^{r})^{a_{r}})$ and let $\Lambda(\lambda,\O)$ be the set consisting of tuples of compositions $\alpha=(\alpha^{(0)}, \alpha^{(1)},\ldots, \alpha^{(r)})$ such that $\lambda=\sum_{i=0}^{r}p^i\alpha^{(i)}$ (not necessarily the $p$-adic expansion of $\lambda$) and $|\alpha^{(i)}|=a_i$ for all $i=0,1,\ldots, r$. It is easy to see that the number $m_{\lambda,\O}$ can be described as \[m_{\lambda,\O}=\sum_{\alpha\in\Lambda(\lambda,\O)}\prod_{i=0}^r \dim_\F M^{\alpha^{(i)}}.\]

To end this section, we give characterisations when an orbit number is nonzero. The following lemma is straightforward following Theorem \ref{Erdmann2}.



\begin{lem}\label{L: prod of dims} Let $\sum^s_{i=0}p^i\lambda(i)$ be the $p$-adic expansion of $\lambda\in\P(n)$. Then \[y_{\lambda,\O_\lambda}=\prod_{i=0}^s\dim_\F Y^{\lambda(i)}\neq 0.\]
\end{lem}

For example, if $\lambda$ has a unique $\lambda(\ell)$ in its $p$-adic expansion with more than one part, then \[y_{\lambda,\O_\lambda}=\dim_\F Y^{\lambda(\ell)}.\] This happens, in particular, when $Y^\lambda$ is a non-projective periodic Young module (see \cite[Corollary 3.3.3]{DHDN}).

Recall that $P_\lambda$ is a fixed Sylow $p$-subgroup of $\sym{\O_\lambda}$ as in Subsection \ref{SS: composition}.

\begin{thm}\label{T:eolambda}
Let $\lambda\in\P(n)$, $\O\in \P^{(p)}(n)$ and $P,E$ be $p$-subgroups such that both $\orbit{\set{n}}{P},\orbit{\set{n}}{E}$ have type $\O$ and $E$ is elementary Abelian. Then the following statements are equivalent.
\begin{enumerate}[(i)]
\item $y_{\lambda,\O}\neq0$.
\item $Y^\lambda{\downarrow_E}$ is not generically free.
\item $P$ is conjugate to a subgroup of $P_\lambda$.
\item $\O$ is rearranged to be a refinement of $\O_{\lambda}$.
\end{enumerate} In any of the cases above, we have $y_{\lambda,\O}\geq y_{\lambda,\O_\lambda}\neq 0$.
\end{thm}
\begin{proof} The equivalence of parts (i), (ii) and (iii) follows from Definition \ref{D: orbit number} and Theorem \ref{Contain}. The equivalence of parts (iii) and (iv) is given by Lemma \ref{L: minimal subgroup}. We now prove the last assertion. Let $Q$ be a conjugate of $P_{\lambda}$ in $\sym{n}$ such that $P$ is a $p$-subgroup of $Q$. Let $\B$ be a $p$-permutation basis of $Y^{\lambda}$ with respect to $Q$. Then $\B(Q)\subseteq\B(P)$. By Corollary \ref{C: EG} and Theorem \ref{T: Yinvarant}, we have \[y_{\lambda,\O}=\dim_{\F}Y^{\lambda}(P)\geq\dim_{\F}Y^{\lambda}(Q)=\dim_{\F}Y^{\lambda}(P_{\lambda})=y_{\lambda,\O_{\lambda}}.\]
\end{proof}

\section{Some computation}\label{S: Some computation}
In this section, we present some equalities among the orbit numbers we have defined in Definition \ref{D: orbit number}. The main results are Theorems \ref{T: Mullineux}, \ref{T: reduction 3} and \ref{T: Reductive5}. 

We present our first results which follows easily from Lemma \ref{L: prod of dims}. Recall that $\m$ is the Mullineux map on $p$-restricted partitions such that $Y^\lambda\otimes\sgn(n)\cong Y^{\m(\lambda)}$ for all $\lambda\in\RP(n)$.


\begin{thm}\label{T: Mullineux} Let $\sum_{i=0}^sp^{i}\lambda(i)$ be the $p$-adic expansion of $\lambda$ and let \[\mu=\sum_{i=0}^sp^{k_i}\m^{\ell_i}(\lambda(i)),\] where $k_0,\ldots,k_s$ are mutually distinct nonnegative integers and, for each $i=0,1,\ldots, s$, $\ell_i$ is either 0 or 1. Then $y_{\lambda,\O_\lambda}=y_{\mu,\O_\mu}$.
\end{thm}
\begin{proof} Notice that $\sum_{i=0}^sp^{k_i}\m^{\ell_i}(\lambda(i))$ is the $p$-adic expansion of $\mu$ since $k_i$'s are all distinct. By Lemma \ref{L: prod of dims}, we have \[y_{\lambda,\O_\lambda}=\prod_{i=0}^s\dim_\F{Y^{\lambda(i)}}=\prod_{i=0}^s\dim_\F{Y^{\m^{\ell_i}(\lambda(i))}}=y_{\mu,\O_\mu}.\]
\end{proof}

Recall that $\beta\bullet \gamma$ denote the concatenation of two compositions $\beta, \gamma$. We need the following lemmas to prove our next result Theorem \ref{T: reduction 3}.

\begin{lem}\label{L: ps partition} Let $m,n,s\in\NN$ such that $p^s>m$. If $\lambda+p^s\mu=\alpha+p^s\beta$ for some $\lambda,\alpha\in\C(m)$ and $\mu,\beta\in\C(n)$ then $\lambda=\alpha$ and $\mu=\beta$.
\end{lem}
\begin{proof} If $\lambda_i>\alpha_i$ for some $i$ then \[p^s>\lambda_i-\alpha_i=p^s(\beta_i-\mu_i)\geq p^s,\] which is a contradiction. Similarly, we must have $\lambda_i\geq\alpha_i$. Therefore $\lambda=\alpha$ and hence $\mu=\beta$.
\end{proof}

\begin{lem}\label{L: Reductive2} Let $\lambda\in\P(m)$, $\mu\in\P(n)$, $\O\in\P^{(p)}(m)$, $\O'\in\P^{(p)}(n)$ and $s\in\NN$ such that $p^s>m$. Then \[m_{\lambda+p^s\mu,\O\bullet p^s\O'}=m_{\lambda,\O}m_{\mu,\O'}.\]
\end{lem}
\begin{proof} Let $\O''=\O\bullet p^{s}\O'$ and let $A=\orbit{\set{m}}{P}$, $B=\orbit{\set{n}}{Q}$ and $C=\orbit{\set{m+p^{s}n}}{R}$. Furthermore, let $P,Q,R$ be $p$-subgroups of $\sym{m},\sym{n},\sym{m+p^sn}$ such that $A,B,C$ have types $\O,\O',\O''$ respectively.  Since $p^s>m$, we may identify the set $C$ with the set $A\cup B$ where an orbit of size $p^i$ in $C$ is identified with an orbit of size $p^i$ in $A$ if $i<s$ and an orbit of size $p^{i-s}$ in $B$ if $i\geq s$. Recall the notation $M_{\lambda,P}$ defined in Subsection \ref{SS: sym}.  To prove the result, we construct a bijection between the sets $X:=M_{\lambda+p^s\mu,R}$ and $Y:=M_{\lambda,P}\times M_{\mu,Q}$. We define $g:Y\to X$ as follows. For each $(\t,\s)\in Y$, let $g(\t,\s)\in X$ be the $(\lambda+p^s\mu)$-tabloid whose $i$th row contains an orbit of $C$ if and only if its corresponding orbit is in $A$ and belongs to the $i$th row of $\t$ or it is in $B$ and belongs to the $i$th row of $\s$. Conversely, we define $f:X\to Y$ as follows. For each $\mathfrak{u}\in X$, let $\t$ be the $\alpha$-tabloid, for some composition $\alpha$ of $m$, whose $i$th row contains an orbit of $A$ if and only if its corresponding orbit in $C$ belongs to the $i$th row of $\mathfrak{u}$. Similarly, we obtain an $\beta$-tabloid $\s$. Since $\alpha+p^s\beta=\lambda+p^s\mu$, by Lemma \ref{L: ps partition}, we have $\alpha=\lambda$ and $\beta=\mu$. Therefore $f(\mathfrak{u})=(\t,\s)\in Y$. Obviously, $f,g$ are inverses of each other. The proof is now complete using Lemma \ref{L: dim of M(P)}.
\end{proof}

\begin{lem}\label{L: red 1} Let $\O\in\P^{(p)}(m)$, $\O'\in\P^{(p)}(n)$ and $s\in\NN$ such that $p^s>m$. If $y_{\tau,\O\bullet p^s\O'}\neq 0$ then $\tau=\nu+p^s\delta$ for some $\nu\in\P(m)$ and $\delta\in\P(n)$.
\end{lem}
\begin{proof} Let $\sum^r_{i=0}p^{i}\tau(i)$ be the $p$-adic expansion of $\tau$. By Theorem \ref{T:eolambda}, $\O'':=\O\bullet p^s\O'$ is a rearrangement of a refinement of $\O_\tau=(1^{|\tau(0)|},\ldots,(p^r)^{|\tau(r)|})$. Let $\nu:=\sum_{i=0}^{s-1}p^{i}\tau(i)$ and $\delta:=\sum_{i=s}^rp^{i-s}\tau(i)$. Notice that both $\nu,\delta$ are partitions. We now show that $|\nu|=m$ and $|\delta|=n$. Since $\O''$ is a rearrangement of a refinement of $\O_\tau$, we have $|\delta|\geq n$. On the other hand, we have $m+p^sn=|\nu|+p^s|\delta|$ and hence $0\leq p^s(|\delta|-n)=m-|\nu|<p^s$. Therefore it forces that $|\delta|=n$ and $|\nu|=m$.
\end{proof}

We are now ready to prove our first reductive formula about orbit numbers.

\begin{thm}\label{T: reduction 3}
Let $\lambda\in \P(m)$, $\mu\in \P(n)$, $\O\in \P^{(p)}(m)$, $\O'\in \P^{(p)}(n)$ and $s,t\in\NN$ such that $p^t\geq p^{s}>m$. Then \[y_{\lambda+p^s\mu,\O\bullet p^{s}\O'}=y_{\lambda+p^t\mu,\O\bullet p^{t}\O'}.\]
\end{thm}
\begin{proof}
Let $\O'':=\O\bullet p^s\O'$ and $\O''':=\O\bullet p^t\O'$. We show our statement by using induction on the set \[X=\{\nu+p^s\delta:\nu\in\P(m),\ \delta\in\P(n)\}\] with respect to the dominance order. When $\nu=(m)$ and $\delta=(n)$, both $Y^{(m+p^sn)}$ and $Y^{(m+p^tn)}$ are trivial modules. So, by Lemma \ref{L: dim of M(P)}, \[y_{(m+p^sn),\O''}=m_{(m+p^sn),\O''}=1=m_{(m+p^tn),\O'''}=y_{(m+p^tn),\O'''}.\]

Suppose that $y_{\nu+p^s\delta,\O''}=y_{\nu+p^t\delta,\O'''}$ for all partitions $\nu+p^s\delta\rhd\lambda+p^s\mu$ where $\nu\in\P(m)$ and $\delta\in\P(n)$. By Equation \ref{Eq: 1}, we have \[m_{\lambda+p^s\mu,\O''}=y_{\lambda+p^s\mu,\O''}+\sum_{\tau\rhd \lambda+p^s\mu}k_{\lambda+p^s\mu,\tau}y_{\tau,\O''}.\] By Lemma \ref{L: red 1}, $y_{\tau,\O''}=0$ unless $\tau=\nu+p^s\delta$ for some uniquely determined $\nu\in\P(m)$ and $\delta\in\P(n)$ (see Lemma \ref{L: ps partition}). By Theorem \ref{T: Reductive}(i), $k_{\lambda+p^s\mu,\nu+p^s\delta}=k_{\lambda,\nu}k_{\mu,\delta}$. Using the observation $k_{\lambda,\nu}k_{\mu,\delta}=0$ unless $\nu\trianglerighteq\lambda$ and $\delta\trianglerighteq\mu$, we deduce that \[m_{\lambda+p^s\mu,\O''}=y_{\lambda+p^s\mu,\O''}+\sum_{\nu\rhd \lambda,\delta\rhd \mu}k_{\lambda,\nu}k_{\mu,\delta}y_{\nu+p^s\delta,\O''}.\] Similarly, we obtain \[m_{\lambda+p^t\mu,\O'''}=y_{\lambda+p^t\mu,\O'''}+\sum_{\nu\rhd \lambda,\delta\rhd \mu}k_{\lambda,\nu}k_{\mu,\delta}y_{\nu+p^t\delta,\O'''}.\] Our result now follows using Lemma \ref{L: Reductive2} and inductive hypothesis.
\end{proof}

We obtain the following immediate consequence.

\begin{cor} Let $t\in\NN$, $\mu\in\P(n)$ and $\O\in\P^{(p)}(n)$. Then $y_{p^t\mu,p^t\O}=y_{\mu,\O}$.
\end{cor}

Next we prove another reductive formula which depends on the shape of the partition (see Theorem \ref{T: Reductive5}) instead of on the size of the partition as in Theorem \ref{T: reduction 3}. The main idea of these two proofs are quite similar but the latter requires slightly different treatment. We need the following two lemmas.

\begin{lem}\label{L: red 2} Let $\lambda\in\P(m)$, $\O\in\P^{(p)}(m)$, $\O'\in\P^{(p)}(n)$ and $s\in\NN$ such that $p^s>\lambda_2$. Then \[m_{\lambda+(p^sn),\O''}=m_{\lambda,\O},\] where $\O''\in\P^{(p)}(m+p^sn)$ is the rearrangement of $\O\bullet p^s\O'$.
\end{lem}
\begin{proof} Let $\O''$ be the rearrangement of $\O\bullet p^s\O'\in \P^{(p)}(m+p^{s}n)$. Let $P,Q$ be $p$-subgroups of $\sym{m},\sym{m+p^sn}$ such that $\orbit{\set{m}}{P},\orbit{\set{m+p^sn}}{Q}$ have types $\O,\O''$ respectively. Let $\t\in M_{\lambda+(p^sn),Q}$. Since $p^s>\lambda_2$, the orbits in $[m+p^sn]/Q$ with sizes larger than or equal to $p^s$ must be assigned to the first row of $\t$. Therefore there is a obvious bijection between $M_{\lambda+(p^sn),Q}$ and $M_{\lambda,P}$. Our claim now follows by using Lemma \ref{L: dim of M(P)}.
\end{proof}

\begin{lem}\label{L: red 3} Let $\lambda\in\P(m)$, $\O\in\P^{(p)}(m)$, $\O'\in\P^{(p)}(n)$ and $s\in\NN$ such that $p^s>m-\lambda_1$. If  $\tau\in \P(m+p^sn)$ such that $\tau\rhd \lambda+(p^sn)$ and $y_{\tau,\O\bullet p^s\O'}\neq 0$ then $\tau=\nu+(p^sn)$ for some uniquely determined partition $\nu\in\P(m)$.
\end{lem}
\begin{proof} Let $\sum^r_{i=0}p^i\tau(i)$ be the $p$-adic expansion of $\tau$ and let $\alpha=\sum^{s-1}_{i=0}p^i\tau(i)$ and $\beta=\sum^r_{i=s}p^{i-s}\tau(i)$ such that $\tau=\alpha+p^s\beta$. We claim that $\beta=(b)$ for some $b\geq n$. Since $y_{\tau,\O\bullet p^s\O'}\neq 0$, we have $\O\bullet p^s\O'$ is a rearrangement of a refinement of $\O_\tau$ and hence $|\beta|\geq |\O'|=n$. Suppose that $\beta_i>0$ for some $i\geq 2$. Since $\tau\rhd \lambda+(p^sn)$, we have $\tau_1\geq \lambda_1+p^sn$. Also, $\tau_i=\alpha_i+p^s\beta_i\geq \alpha_i+p^s$. Therefore \[p^s\leq\alpha_i+p^s\leq \tau_i\leq (m+p^sn)-\tau_1\leq m-\lambda_1<p^s,\] which is a contradiction. This shows that $\beta$ has at most one part and hence $\beta=(b)$ for some $b\geq n$. Let $\nu=\alpha+p^s(b-n)$. Then $\tau=\nu+(p^sn)$.
\end{proof}

We are now ready to prove the second reductive formula about orbit numbers.

\begin{thm}\label{T: Reductive5}
Let $\lambda\in \P(m)$ such that $m-\lambda_{1}<p^s$ for some $s\in\NN$, let $\O\in\P^{(p)}(m)$ and let $\O'\in \P^{(p)}(n)$. Then \[y_{\lambda+(p^sn),\O''}=y_{\lambda,\O},\] where $\O''\in\P^{(p)}(m+p^sn)$ is the rearrangement of $\O\bullet p^s\O'$.
\end{thm}
\begin{proof} We argue by using induction with respect to the dominance order on the set \[X=\{\nu\in\P(m):m-\nu_1<p^s\}.\] W hen $\nu=(m)\in X$, we have $y_{(m)+(p^sn),\O''}=1=y_{(m),\O}$. Suppose that $y_{\nu+(p^sn),\O''}=y_{\nu,\O}$ for all $\lambda\lhd\nu\in X$. By Equation \ref{Eq: 1} and Lemma \ref{L: red 3}, we have
\begin{align*}
m_{\lambda+(p^sn),\O''}&=y_{\lambda+(p^sn),\O''}+\sum_{\tau\rhd \lambda+(p^sn)}k_{\lambda+(p^sn),\tau}y_{\tau,\O''}\\
&=y_{\lambda+(p^sn),\O''}+\sum_{\nu\rhd \lambda}k_{\lambda+(p^sn),\nu+(p^sn)}y_{\nu+(p^sn),\O''}.
\end{align*} By Theorem \ref{T: Reductive}(ii), since $\lambda_2\leq m-\lambda_1<p^s$, we have $k_{\lambda+(p^sn),\nu+(p^sn)}=k_{\lambda,\nu}$. By inductive hypothesis, we have \[m_{\lambda+(p^sn),\O''}=y_{\lambda+(p^sn),\O''}+\sum_{\nu\rhd \lambda}k_{\lambda,\nu}y_{\nu,\O}.\] The proof is now complete by using Equation \ref{Eq: 1} $m_{\lambda,\O}=y_{\lambda,\O}+\sum_{\nu\rhd \lambda}k_{\lambda,\nu}y_{\nu,\O}$ and Lemma \ref{L: red 2}.
\end{proof}

\section{Two-part partitions}\label{S: two part}
In the final section, we provide some explicit calculation about the orbit numbers $y_{\lambda,\O}$ when $\lambda$ is a two-part partition. We begin with the following proposition.

\begin{prop} Let $\O=(1^{a_{0}}, p^{a_{1}},\ldots, (p^{r})^{a_{r}})\in\P^{(p)}(n)$. If $a_0=0$, then $y_{(n-1,1),\O}=0$. If $a_0\neq 0$ then
\[y_{(n-1,1),\O}=\begin{cases}
             a_0, &\text{ if $p\mid n$,}\\
              a_0-1, &\text{ if $p\nmid n$}.

\end{cases}\]
\end{prop}
\begin{proof} Let $P$ be a $p$-subgroup such that $\orbit{\set{n}}{P}$ has type $\O$. It is well-known that $M^{(n-1,1)}$ is isomorphic to $Y^{(n-1,1)}$ if $p$ divides $n$ and $Y^{(n)}\oplus Y^{(n-1,1)}$ otherwise. By Lemma \ref{L: dim of M(P)}, $\dim_\F M^{(n-1,1)}(P)$ is the number of ways to insert the orbits with size one that are in $\set{n}/P$ into the second row of $(n-1,1)$, i.e., $\dim_\F M^{(n-1,1)}(P)=a_0$. If $a_0=0$, since $0\leq y_{(n-1,1),\O}\leq \dim_\F M^{(n-1,1)}(P)=0$, then $y_{(n-1,1),\O}=0$. If $a_0\neq 0$, then
\[y_{(n-1,1),\O}= \dim_\F M^{(n-1,1)}(P)-k_{(n),(n-1,1)}\dim_\F Y^{(n)}(P)=a_0-k_{(n),(n-1,1)},\] where $k_{(n),(n-1,1)}$ is 1 if $p\nmid n$ and 0 otherwise.
\end{proof}

Next, we compute $y_{\lambda,\O_\lambda}$ when $\lambda$ is a two-part partition. We need the following two lemmas.


\begin{lem}\label{L:pweight}
Any $p$-restricted two-part partition has $p$-weight either 0 or 1. Furthermore, a $p$-restricted partition $(a,b)$ has $p$-weight 1 if and only if $a-b< p-1$ and $a+1\geq p$. In this case, the $p$-core is $(b-1,a+1-p)$.
\end{lem}
\begin{proof}
Let $\lambda=(a,b)$ which is a $p$-restricted two-part partition, i.e., $a-b<p$ and $b<p$. Suppose that the $p$-weight of $\lambda$ is not zero. It is clear from the Young diagram of $\lambda$ that it is equivalent to $a-b<p-1$ and $a+1\geq p$. In this case, after removing one $p$-hook from $\lambda$, the remaining partition is $(b-1,a+1-p)$. However, $(b-1)+1<p$, so $(b-1,a+1-p)$ has $p$-weight 0.
\end{proof}

The weight one blocks of symmetric groups are Morita equivalent by \cite[Theorem 1]{Scopes}.For $p\geq 3$, the decomposition matrix of the principal block of $\F\sym{p}$ has been described by Peel in \cite{Peel} (see also \cite[Theorem 24.1]{GJ1}). Let $b$ be the weight one block of $\F\sym{n}$ labelled by a $p$-core $\nu$ and let \[\mu^{(0)}\rhd\mu^{(1)}\rhd\cdots\rhd\mu^{(p-1)}\] be all the partitions occurring in $b$. Notice that $\mu^{(i)}$ is $p$-restricted for each $i=1,2,\ldots, p-1$ and $\mu^{(0)}=\nu+(p)$. Taking the conjugates, we have $(\mu^{(0)})'\lhd(\mu^{(1)})'\lhd\cdots\lhd(\mu^{(p-1)})'$. By the Brauer reciprocity, if $\mu$ is $p$-restricted, we have that the multiplicity of the ordinary irreducible character $\chi^\lambda$ in $\ch(Y^\mu)$ is the decomposition number $d_{\lambda',\mu'}$, i.e., the multiplicity of the simple module $D^{\mu'}$ in the composition series of $S^{\lambda'}$. In particular, we have \[\dim_\F Y^{\mu^{(i)}}=\dim_\F S^{\mu^{(i)}}+\dim_\F S^{\mu^{(i-1)}}.\]

Suppose that $\lambda=(a,b)$ is a $p$-restricted partition. Suppose first that $p$ is odd. If the $p$-weight of $\lambda$ is zero then $Y^\lambda\cong S^\lambda$. If the $p$-weight of $\lambda$ is one then, by Lemma \ref{L:pweight}, in our discussion above, $\mu^{(1)}=\lambda$ and $\mu^{(0)}=\kappa_p(\lambda)+(p)=(p+b-1,a+1-p)$.  Suppose now that  $p=2$. We have that $(a,b)$ is either $(2,1)$ or $(1,1)$. In this case, $Y^{(2,1)}\cong S^{(2,1)}$ and $M^{(1,1)}\cong Y^{(1,1)}$. We have obtained the following lemma.

\begin{lem}\label{L: dim of two part}
For a $p$-restricted two-part partition $\lambda$, we have
\[\dim_\F{Y^\lambda}=\begin{cases}
\dim_\F{S^\lambda},& \text{ if $\lambda=\kappa_p(\lambda)$,}\\
\dim_\F{S^\lambda}+\dim_\F{S^{\kappa_p(\lambda)+(p)}},& \text{ if $\lambda\neq \kappa_p(\lambda)$.}
\end{cases}
\]
\end{lem}

Now we can give a description for the orbit numbers $y_{\lambda,\O_\lambda}$ when $\lambda$ is two-part partition.

\begin{prop}
Let $\lambda=(a,b)$ be a two-part partition and let the $p$-adic sums of the numbers $a-b$ and $b$ be $\sum_{i\geq0}x_ip^{i}$ and $\sum_{i\geq0}y_ip^{i}$, respectively.
Then
\[y_{\lambda,\O_\lambda}=\prod_{i\geq 0} \left ({x_i+2y_i-1\choose y_i}+\delta(x_i,y_i){x_i+2y_i-1\choose x_i+y_i+1-p}\right ),\]
where $\delta$ is the function defined as
\[\delta(x,y)=\begin{cases}
              1, &\text{ if $x<p-1$ and $x+y+1\geq p$,}\\
              0, &\text{ otherwise.}
\end{cases}
\]

\end{prop}
\begin{proof} Notice that the $p$-adic expansion of $(a,b)$ is $\sum_{i\geq 0}p^i(x_i+y_i,y_i)$. By Lemma \ref{L:pweight}, the partition $(x_i+y_i,y_i)$ has $p$-weight 1 if and only if $x_i<p-1$ and $x_i+y_i+1\geq p$. In this case, \[\kappa_p(\lambda(i))+(p)=(p+y_i-1,x_i+y_i+1-p).\] Otherwise, the $p$-weight of $(x_i+y_i,y_i)$ is zero. By Lemma \ref{L: dim of two part}, we have
\[\dim_\F Y^{\lambda(i)}=\dim_\F{S^{(x_i+y_i,y_i)}}+\delta(x_i,y_i)\dim_\F{S^{(y_i+p-1,x_i+y_i+1-p)}}.\] The proof is now complete using Lemma \ref{L: prod of dims} and Hook Formula for the dimension of a Specht module.
\end{proof}

\end{document}